\documentclass[preprint,12pt]{elsarticle1}

\usepackage{subfig}
\usepackage{graphicx}
\usepackage{caption,color}   
\usepackage{amssymb}
\usepackage{amsthm}
\usepackage{amsmath}
\usepackage{epic}
\usepackage{setspace}
\usepackage{float}
\usepackage{multirow}
\usepackage{hyperref}
\newtheorem{thm}{Theorem}[section]
\newtheorem{cor}[thm]{Corollary}

\newtheorem{lem}[thm]{Lemma}

\newtheorem{defi}[thm]{Definition}

\numberwithin{equation}{section}
\journal{}

\begin{document}
\begin{spacing}{1.15}
\begin{frontmatter}
\title{The ordering of hypertrees and unicyclic hypergraphs by the traces of $\mathcal{A}_{\alpha}$-tensor}

\author{Jueru Liu}
\author{Lizhu Sun}\ead{lizhusun@hrbeu.edu.cn}
\author{Changjiang Bu}
\address{School of Mathematical Sciences, Harbin Engineering University, Harbin 150001, PR China}

\begin{abstract}
For a real number $\alpha\in[0,1]$ and a $k$-uniform hypergraph $\mathcal{H}$, $\mathcal{A}_{\alpha}(\mathcal{H})=\alpha\mathcal{D}(\mathcal{H})+(1-\alpha)\mathcal{A}(\mathcal{H})$ is called the $\mathcal{A}_{\alpha}$-tensor of $\mathcal{H}$, where $\mathcal{D}(\mathcal{H})$ and $\mathcal{A}(\mathcal{H})$ are the degree tensor and adjacency tensor of $\mathcal{H}$, respectively.
The sum of the $d$-th powers of all eigenvalues of $\mathcal{A}_{\alpha}(\mathcal{H})$ is called the $d$-th order $\mathcal{A}_{\alpha}$-spectral moment of $\mathcal{H}$, which is equal to the $d$-th order trace of $\mathcal{A}_{\alpha}(\mathcal{H})$.
In this paper, some hypergraphs are ordered lexicographically by their $\mathcal{A}_{\alpha}$-spectral moments in non-decreasing order.
The first, the second, the last and the second last hypergraphs among all $k$-uniform linear unicyclic hypergraphs and hypertrees are characterized, respectively.
We give the first and the last hypergraphs among all $k$-uniform linear unicyclic hypergraphs with given grith,
and characterize the last hypertree among all $k$-uniform hypertrees with given diameter.
Furthermore, we determine some extreme values of the $\mathcal{A}_{\alpha}$-spectral moments for hypertrees and linear unicyclic hypergraphs, respectively.
\end{abstract}

\end{frontmatter}

\section{Introduction}
In 1981, Cvetkovi{\'c} \cite{0Some} proposed that eigenvalues can be used for classification and ordering of graphs, which is first used in \cite{firstordergraph}.
Since small changes in eigenvalues will restrict changes to the relevant structural invariants, graphs can be ordered lexicographically by their spectra \cite{cvetkovic2009introduction}.
The $d$-th order spectral moment of a graph is the sum of the $d$-th powers of all its eigenvalues.
In 1984, the spectral moments were first used to ordering graphs lexicographically \cite{sixvertices}.
Ordering graphs by their spectral moments is a slight modifcation of ordering by eigenvalues \cite{sixvertices}.
In 1987, Cvetkovi{\'c} and Rowlinson \cite{Drago1987Spectra} defined the lexicographical ordering of graphs by spectral moments, and characterized the first and the last graphs among all trees and all unicyclic graphs, respectively.
In \cite{WOS:000283893700001}, the last $\lfloor\frac{d}{2}\rfloor+1$ trees among all trees with diameter $d$ were given.
In \cite{diameter}, the last $d+\lfloor\frac{d}{2}\rfloor-2$ graphs among all unicyclic graphs with diameter $d$ were characterized.
Other works on the ordering of graphs by spectral moments can be referred to \cite{WOS:000301826900014,WOS:000486390400051,tree,WOS:000446809500027}.

In 2005, Qi \cite{ref2} and Lim \cite{ref1} proposed the definition of tensor eigenvalues independently.
The eigenvalues of tensors and related problems have received extensive attention \cite{WOS:000346407800007,Fan2019AMS,changanSIAM2022,LHH2024,Nikiforov2017,MP}, especially the trace of tensors.
In 2011, Morozov and Shakirov \cite{ref6} gave an expression for the trace of a tensor.
In 2013, Hu et al. \cite{ref7} proved that the $d$-th order trace of a tensor is equal to the sum of the $d$-th powers of all its eigenvalues.
In 2015, Shao et al. \cite{ref8} gave a formula for the tensor trace in terms of some digraph parameters.
The $d$-th order spectral moment of a uniform hypergraph is the sum of $d$-th powers of all its eigenvalues.
The characteristic polynomial coefficients and all eigenvalues of a uniform hypergraph can be determined by all orders of its spectral moments \cite{ref11,ref23,ref10,Duan2023doublestar}.
In 2021, Clark and Cooper \cite{ref10} expressed the spectral moments of uniform hypergraphs by the number of Veblen hypergraphs and obtained the ``Harary-Sachs" theorem for hypergraphs.
In 2024, Chen, van Dam and Bu \cite{ref23} expressed the spectral moments of power hypergraphs by the counts of parity-closed walks in graphs, and gave the characteristic polynomials and the multiplicity of the spectral radius of power hypergraphs \cite{ref23}.
The spectral moments of a uniform hypergraph can also be used to characterize its structure \cite{ref16,ref31} and parameters \cite{ref35,ref22}.
In 2023, the lexicographical ordering of hypergraphs via spectral moments was investigated \cite{sorder}, which is called the $S$-order of hypergraphs, and the first and the last hypergraphs in $S$-order among all hypertrees and all linear unicyclic hypergraphs were characterized, respectively.

In 2017, Nikiforov \cite{alphaofgraph} proposed the $A_{\alpha}$-matrix of a graph, which is a linear combination of degree matrix and adjacency matrix, and some properties of $A_{\alpha}$-matrix were given.
In 2020, Lin, Guo and Zhou \cite{alphaofhypergraph} generalized the $A_{\alpha}$-spectra of graphs to hypergraphs, proposed and investigated the $\mathcal{A}_{\alpha}$-tensor of uniform hypergraphs.
The research on $\mathcal{A}_{\alpha}$-spectra of graphs and hypergraphs has attracted the attention of scholars \cite{WOS:000547249300008,WOS:000443667300012,WOS:000864637600006,WOS:001490552500001,WOS:000699906300018}.

The $d$-th order $\mathcal{A}_{\alpha}$-spectral moment of a $k$-uniform hypergraph is the sum of the $d$-th powers of all eigenvalues of its $\mathcal{A}_{\alpha}$-tensor.
Note that the $\mathcal{A}_{\alpha}$-tensor of a $k$-uniform hypergraph is a linear combination of its degree tensor and adjacency tensor for $\alpha\in(0,1)$, then the $\mathcal{A}_{\alpha}$-spectral moments cannot be directly obtained from its spectral moments.
In this paper, we use Shao et al.'s tensor trace formula to give some expressions for the $\mathcal{A}_{\alpha}$-spectral moments of $k$-uniform hypergraphs.
We investigate the lexicographical ordering of hypergraphs by their $\mathcal{A}_{\alpha}$-spectral moment for $\alpha\in(0,1)$, which is abbreviated as $S_{\alpha}$-order.
And some extreme values of the $\mathcal{A}_{\alpha}$-sepctral moments are obtained.
When $k=2$, some conclusions in this paper are also new results for the $A_{\alpha}$-spectral moments of graphs.

This paper is organized as follows.
In Section 2, some notations and concepts about tensors and hypergraphs are introduced.
In Section 3, some expressions for the $\mathcal{A}_{\alpha}$-spectral moments of uniform hypergraphs are given.
In Section 4, we introduce some operations of moving hyperedges or hyperpaths on uniform hypergraphs and prove that the $S_{\alpha}$-order of hypergraphs is monotonic on these transformations.
In Section 5, we characterize the first, the second, the last and the second last hypergraphs in $S_{\alpha}$-order among all $k$-uniform linear unicyclic hypergraphs, and give the first and the last hypergraphs among all $k$-uniform linear unicyclic hypergraphs with given grith.
In Section 6, we give the first, the second, the last and the second last hypertrees in $S_{\alpha}$-order among all $k$-uniform hypertrees, and the last hypertree in $S_{\alpha}$-order among all $k$-uniform hypertrees with given diameter is characterized.
In Section 7, we determine the hypergraphs which have the largest and the second-largest $2$-nd order $\mathcal{A}_{\alpha}$-spectral moments among all linear unicyclic hypergraphs and hypertrees, respectively. And the hypergraphs with the smallest and the second-smallest $k+2$-nd $\mathcal{A}_{\alpha}$-spectral moments among all linear unicyclic hypergraphs and hypertrees are determined, respectively.

\section{Preliminaries}
In this section, some notations and concepts about tensors and hypergraphs are introduced.
For a positive integer $n$, let $[n]=\{1,2,\ldots,n\}$ and $[n]^{k}=\{i_1i_2\cdots i_k|~ i_j\in[n]~ \mathrm{for}~ j=1,2,\ldots,k\}$.
A $k$-order $n$-dimensional complex tensor $\mathcal{A}=(a_{i_1 i_2 \cdots i_k})$ is a multidimensional array with $n^k$ entries on complex number field $\mathbb{C}$, where $i_1i_2\cdots i_k\in[n]^k$.
Denote the set of $n$-dimensional complex vectors and the set of $k$-order $n$-dimensional complex tensors by $\mathbb{C}^{n}$ and $\mathbb{C}^{[k,n]}$, respectively.
For $\mathcal{A}=(a_{i_1 i_2 \cdots i_k})\in\mathbb{C}^{[k,n]}$ and $x=(x_1,\ldots,x_n)^{\mathsf{T}}\in\mathbb{C}^{n}$, $\mathcal{A}x^{k-1}$ is a vector in $\mathbb{C}^{n}$ whose $i$-th component is $$\left( \mathcal{A}x^{k-1} \right)_{i}=\sum\limits_{i_2,\ldots,i_k=1}^{n}a_{i i_2 \cdots i_k}x_{i_2}\cdots x_{i_k}.$$
A number $\lambda\in\mathbb{C}$ is called an \textit{eigenvalue} of $\mathcal{A}$ if there exists a nonzero vector $x\in\mathbb{C}^{n}$ such that $$\mathcal{A}x^{k-1}=\lambda x^{[k-1]},$$ where $x^{[k-1]}=(x_1^{k-1},\ldots,x_n^{k-1})^{\mathsf{T}}$ \cite{ref1,ref2}.
And the number of eigenvalues of $\mathcal{A}$ is $n(k-1)^{n-1}$.

In \cite{ref6}, the $d$-th order \textit{trace} $\mathrm{Tr}_{d}(\mathcal{A})$ of a tensor $\mathcal{A}\in\mathbb{C}^{[k,n]}$ is expressed as
$$\mathrm{Tr}_{d}(\mathcal{A})=(k-1)^{n-1}\sum\limits_{d_1+\cdots+d_n=d}\prod\limits_{i=1}^{n}\frac{1}{(d_i(k-1))!}\left(\sum\limits_{\alpha_i\in[n]^{k-1}}(\mathcal{A})_{i\alpha_i}\frac{\partial}{\partial a_{i\alpha_i}}\right)^{d_{i}}\mathrm{tr}(A^{d(k-1)}),$$
where $A=(a_{ij})$ is an $n\times n$ auxiliary matrix, $d_1\ldots,d_n$ are nonnegative integers and $\frac{\partial}{\partial a_{i\alpha_i}}=\frac{\partial}{\partial a_{i i_2}}\cdots\frac{\partial}{\partial a_{i i_k}}$ for $\alpha_{i}=i_2\cdots i_k$.

Shao et al. \cite{ref8} gave a graph theoretical formula for the trace $\mathrm{Tr}_{d}(\mathcal{A})$ of a tensor $\mathcal{A}$.
For a positive integer $d$, let $$\mathcal{F}_{d}=\{(i_{1}\alpha_{1},\ldots,i_{d}\alpha_{d})|\ 1\le i_{1}\le\cdots\le i_{d}\le n,~ \alpha_{j}\in[n]^{k-1},~ j=1,\ldots,d\}.$$
For $f=(i_{1}\alpha_{1},\ldots,i_{d}\alpha_{d})\in\mathcal{F}_{d}$, where $i_j\alpha_j=i_ji_{j_2}\cdots i_{j_k},~ j\in[d]$,
let $E_{j}(f)=\{(i_j,i_{j_2}),\ldots,(i_j,i_{j_k})\}$ be the multi-set of arcs from $i_j$ to $i_{j_2},\ldots,i_{j_k}$, and let $E(f)=\bigcup_{j=1}^{d}E_j(f)$ be a arc multi-set.
Let the multi-digraph $D(f)=(V(f),E(f))$, where $V(f)=V(E(f))$.
Let $b(f)$ be the product of the factorials of the multiplicities of all arcs of $D(f)$, let $c(f)$ be the product of the factorials of the outdegrees of all vertices of $D(f)$, and let $W(f)$ be the set of all Euler tours of $D(f)$.
For a $k$-order $n$-dimensional tensor $\mathcal{A}$, let $\pi_{f}(\mathcal{A})=\prod_{j=1}^{d} (\mathcal{A})_{i_j\alpha_j}$.

\begin{lem}\label{Shao}\cite{ref8}
		Let $\mathcal{A}$ be a $k$-order $n$-dimensional tensor. Then
		\begin{equation*}\label{shaodegongshi}
			\textnormal{Tr}_{d}(\mathcal{A})=(k-1)^{n-1} \sum_{f \in \mathcal{F}_{d}} \frac{b(f)}{c(f)} \pi_{f}(\mathcal{A})|W(f)|.
		\end{equation*}
	\end{lem}	

A hypergraph $\mathcal{H}$ is called \textit{k-uniform} if every hyperedge of $\mathcal{H}$ contains exactly $k$ vertices.
Let $C_{m}^{(k)}$, $P_{m}^{(k)}$ and $S_{m}^{(k)}$ denote the $k$-uniform hypercycle, hyperpath and hyperstar with $m$ hyperedges, respectively.
For a $k$-uniform hypergraph $\mathcal{H}$ with $n$ vertices, its \textit{adjacency tensor} $\mathcal{A}(\mathcal{H})=(a_{i_1 i_2 \cdots i_k})$ is a $k$-order $n$-dimensional tensor \cite{ref3}, where
		$$a_{i_1 i_2 \cdots i_k}=\left\{\begin{array}{cl}
			\frac{1}{(k-1)!},&  \textnormal{if}\ \{i_1,i_2,\ldots,i_k\}\in E(\mathcal{H}),\\
			0,& \textnormal{otherwise}.
		\end{array}\right.$$
When $k=2$, $\mathcal{A}(\mathcal{H})$ is the adjacency matrix of the graph $\mathcal{H}$.
The tensor $\mathcal{Q}(\mathcal{H})=\mathcal{D}(\mathcal{H})+\mathcal{A}(\mathcal{H})$ is the \textit{signless Laplacian tensor} of $\mathcal{H}$ \cite{ref5}, where $\mathcal{D}(\mathcal{H})$ is the diagonal tensor of vertex degrees.

The \textit{$\mathcal{A}_{\alpha}$-tensor} of a $k$-uniform hypergraph $\mathcal{H}$ is defined as \cite{alphaofhypergraph}
$$\mathcal{A}_{\alpha}(\mathcal{H})=\alpha\mathcal{D}(\mathcal{H})+(1-\alpha)\mathcal{A}(\mathcal{H}),~ 0\le\alpha\le1.$$
Obviously, $$\mathcal{A}(\mathcal{H})=\mathcal{A}_{0}(\mathcal{H}),~ \mathcal{D}(\mathcal{H})=\mathcal{A}_{1}(\mathcal{H}),~ \mathcal{Q}(\mathcal{H})=2\mathcal{A}_{\frac{1}{2}}(\mathcal{H}).$$
In this paper, we order some hypergraphs lexicographically by their $\mathcal{A}_{\alpha}$-spectral moments in non-decreasing order.

\begin{defi}
For two $k$-uniform hypergraphs $\mathcal{H}_1$ and $\mathcal{H}_2$ with $n$ vertices,
if there exists an integer $d\in[n(k-1)^{n-1}-1]$ such that $\mathrm{Tr}_t(\mathcal{A}_{\alpha}(\mathcal{H}_1))=\mathrm{Tr}_{t}(\mathcal{A}_{\alpha}(\mathcal{H}_2))$ for $t=0,1,\ldots,d-1$ and $\mathrm{Tr}_d(\mathcal{A}_{\alpha}(\mathcal{H}_1))<\mathrm{Tr}_{d}(\mathcal{A}_{\alpha}(\mathcal{H}_2))$, then $\mathcal{H}_{1}$ comes before $\mathcal{H}_{2}$ in $S_{\alpha}$-order, denoted by $\mathcal{H}_{1} \prec_{\alpha} \mathcal{H}_{2}$.
If $\mathrm{Tr}_t(\mathcal{A}_{\alpha}(\mathcal{H}_1))=\mathrm{Tr}_{t}(\mathcal{A}_{\alpha}(\mathcal{H}_2))$ for $t=0,1,\ldots,n(k-1)^{n-1}-1$, then $\mathcal{H}_1=_{\alpha}\mathcal{H}_2$.
And $\mathcal{H}_{1}\preceq_{\alpha}\mathcal{H}_{2}$ if $\mathcal{H}_{1}\prec_{\alpha}\mathcal{H}_{2}$ or $\mathcal{H}_{1}=_{\alpha}\mathcal{H}_{2}$.
\end{defi}

It is obvious that two tensors have the same spectrum if and only if their trace of any order is equal.
Hence, for two $k$-uniform hypergraphs $\mathcal{H}_1$ and $\mathcal{H}_2$ with $n$ vertices, $\mathcal{H}_{1}=_{\alpha}\mathcal{H}_{2}$ if and only if they have the same $\mathcal{A}_{\alpha}$-spectrum.


\section{$\mathcal{A}_{\alpha}$-spectral moments of hypergraphs}
For a $k$-uniform hypergraph $\mathcal{H}$ with $n$ vertices, it is known that $\mathrm{Tr}_{0}(\mathcal{A}_{\alpha}(\mathcal{H}))$ is equal to the number of eigenvalues of $\mathcal{A}_{\alpha}(\mathcal{H})$, i.e., $\mathrm{Tr}_{0}(\mathcal{A}_{\alpha}(\mathcal{H}))=n(k-1)^{n-1}$.
In this section, some expressions of $\mathcal{A}_{\alpha}$-spectral moment are given.

Given $f=(i_{1}\alpha_{1},\ldots,i_{d}\alpha_{d})\in\mathcal{F}_{d}$, where $i_j\alpha_j=i_ji_{j_2}\cdots i_{j_k}$ for $j\in[d]$, we construct a $k$-uniform multi-hypergraph $\mathcal{H}_f=(V(\mathcal{H}_f),E(\mathcal{H}_f))$, where $V(\mathcal{H}_f)=V(f)$ and $E(\mathcal{H}_{f})=\{\{i_j,i_{j_2},\ldots,i_{j_k}\}^{a}|~ \{i_j,i_{j_2},\ldots,i_{j_k}\}~ \mathrm{forms}~ \mathrm{a}~ \mathrm{hyperedge}~ \mathrm{and}~ a\ \mathrm{is}~ \mathrm{the}~ \mathrm{number}~ \mathrm{of}\\ \mathrm{elements}~ \mathrm{in}~ f~ \mathrm{that}~ \mathrm{have}~ \mathrm{the}~ \mathrm{same}~ \mathrm{indices}~ \mathrm{as}~ i_j\alpha_j\}$.
If the hypergraph obtained by removing duplicate hyperedges of $\mathcal{H}_{f}$ is a subhypergraph of a $k$-uniform hypergraph $\mathcal{H}$, then $\mathcal{H}_{f}$ is called an \textit{infragraph} of $\mathcal{H}$.

A \textit{Veblen hypergraph} is a $k$-uniform, $k$-valent (i.e., the degree of every vertex is a multiple of $k$) multi-hypergraph \cite{ref10}.
For a $k$-uniform hypergraph $\mathcal{H}$, let $\mathcal{V}_{h}(\mathcal{H})$ be the set of all connected Veblen infragraphs of $\mathcal{H}$ with at most $h$ hyperedges.
For $H\in\mathcal{V}_{h}(\mathcal{H})$, let $\mathcal{F}_{d}(H)=\{f\in\mathcal{F}_d|\ \mathcal{H}_f\cong H\}$.

By using Shao et al.'s tensor trace formula \cite{ref8}, the expression for $\mathcal{A}_{\alpha}$-spectral moment of any order of a $k$-uniform hypergraph is given as follows.
For a $k$-uniform hypergraph $\mathcal{H}$ with degree sequence $d_1,d_2,\ldots,d_n$ and a nonnegative integer $s$, let $$\phi(s)=(k-1)^{n-1}\alpha^{s}\sum_{i=1}^{n}d_{i}^{s}.$$

\begin{thm}\label{T}
Let $k\ge2$ and $\mathcal{H}$ be a $k$-uniform hypergraph with degree sequence $d_1,d_2,\ldots,d_n$.
Then
\begin{equation*}
\begin{aligned}
&\mathrm{Tr}_{d}(\mathcal{A}_{\alpha}(\mathcal{H}))=\\& \phi(d)+(1-\alpha)^{d}\mathrm{Tr}_{d}(\mathcal{A}(\mathcal{H}))+d(k-1)^{n}\sum\limits_{H\in\mathcal{V}_{d-1}(\mathcal{H})}\sum\limits_{f\in\mathcal{F}_{d}(H)}\frac{\tau(f)\pi_{f}(\mathcal{A}_{\alpha}(\mathcal{H}))}{\prod\limits_{v\in V(f)}d^{+}(v)},
\end{aligned}
\end{equation*}
where $\tau(f)$ is the number of rooted spanning trees of the multi-digraph $D(f)$ with a specified root.
\end{thm}

\begin{proof}
Let $\mathcal{A}_{\alpha}(\mathcal{H})$ be the $\mathcal{A}_{\alpha}$-tensor of $\mathcal{H}$. From Lemma \ref{Shao}, the $d$-th order $\mathcal{A}_{\alpha}$-spectral moment of $\mathcal{H}$
$$\textnormal{Tr}_{d}(\mathcal{A}_{\alpha}(\mathcal{H}))=(k-1)^{n-1}\sum_{f\in \mathcal{F}_{d}}\frac{b(f)}{c(f)}\pi_{f}(\mathcal{A}_{\alpha}(\mathcal{H}))|W(f)|.$$

Since $\mathcal{A}_{\alpha}(\mathcal{H})=\alpha\mathcal{D}(\mathcal{H})+(1-\alpha)\mathcal{A}(\mathcal{H})$ is a linear combination of degree tensor and adjacency tensor of $\mathcal{H}$,
we know that  $\pi_f(\mathcal{A}_{\alpha}(\mathcal{H}))\ne0$ if and only if for all $j\in[d]$ either $i_j\alpha_j=i_ji_j\cdots i_j$ or $i_j\alpha_j\in E(\mathcal{H})$ for $f=(i_1\alpha_1,\ldots,i_d\alpha_d)\in\mathcal{F}_{d}$.
There are only three cases for the elements of $\mathcal{F}_d$ such that $\pi_{f}(\mathcal{A}_{\alpha}(\mathcal{H}))\ne0$.
Let $\mathcal{F}_{d}^{*}=\{f\in\mathcal{F}_{d}|\ \pi_{f}(\mathcal{A}_{\alpha}(\mathcal{H}))\ne0\}=\mathcal{F}_{d}^{(1)}\cup\mathcal{F}_{d}^{(2)}\cup\mathcal{F}_{d}^{(3)}$, where $\mathcal{F}_{d}^{(i)}~(i=1,2,3)$ are shown as follows.\\
(1) $\mathcal{F}_{d}^{(1)}=\{(i_{1}\alpha_{1},\ldots,i_{d}\alpha_{d})|\ i_j\alpha_{j}=i_ji_j\cdots i_j\ \textnormal{for}\ j\in[d]\}$,\\
(2) $\mathcal{F}_{d}^{(2)}=\{(i_{1}\alpha_{1},\ldots, i_{d}\alpha_{d})|\ i_j\alpha_j\in E(\mathcal{H})\ \textnormal{for}\ j\in[d]\}$,\\
(3) $\mathcal{F}_{d}^{(3)}=\{(i_{1}\alpha_{1},\ldots,i_{d}\alpha_{d})|\  \textnormal{there exists}\ S\subset[d]\ \textnormal{such that}\ i_s\alpha_s=i_si_s\cdots i_s\ \textnormal{for}\ s\in S,\ i_t\alpha_t\in E(\mathcal{H})\ \textnormal{for}\ t\notin S\}$.

By the definitions of $\mathcal{F}_{d}^{(i)}~(i=1,2,3)$, it is obvious that $\mathcal{F}_{d}^{(j)}\cap\mathcal{F}_{d}^{(l)}=\emptyset$ for $1\le j<l\le3$.
Let $$\omega_i=\sum\limits_{f\in\mathcal{F}_{d}^{(i)}}\frac{b(f)}{c(f)}\pi_{f}(\mathcal{A}_{\alpha}(\mathcal{H}))|W(f)|,\ i=1,2,3.$$
Then
\begin{equation}\label{laplacian3terms}
\textnormal{Tr}_{d}(\mathcal{A}_{\alpha}(\mathcal{H}))=(k-1)^{n-1}\sum\limits_{i=1}^{3}\omega_i.
\end{equation}

Next, we consider $f=(i_1\alpha_1,\ldots,i_d\alpha_d)\in\mathcal{F}_{d}^{*}$ for which $|W(f)|\neq0$.
From the proofs of Lemma 3.1 and Lemma 3.2 in \cite{L}, $\{f\in\mathcal{F}_{d}^{(i)}|~ |W(f)|\ne0\}$ and $\omega_i$ are shown as follows ($i=1,2,3$).\\
(1) $\{f\in\mathcal{F}_{d}^{(1)}|~ |W(f)|\ne0\}=\{(vv\cdots v,\ldots,vv\cdots v)|~ v\in[n]\}$, then
\begin{equation}\label{w1}
\omega_{1}=\alpha^{d}\sum\limits_{v=1}^{n}d_{v}^{d}=(k-1)^{1-n}\phi(d).
\end{equation}
(2) $\{f\in\mathcal{F}_{d}^{(2)}|~ |W(f)|\ne0\}=\bigcup_{H\in\mathcal{V}_{d}^{*}(\mathcal{H})}\mathcal{F}_d(H)$, where $\mathcal{V}_{d}^{*}(\mathcal{H})$ is the set of all connected Veblen infragraphs of $\mathcal{H}$ with $d$ hyperedges, then
\begin{equation}\label{w2}
\omega_2=(k-1)^{1-n}(1-\alpha)^{d}\textnormal{Tr}_{d}(\mathcal{A}_{\mathcal{H}}).
\end{equation}
(3) $\{f\in\mathcal{F}_{d}^{(3)}|~ |W(f)|\ne0\}=\bigcup_{H\in\mathcal{V}_{d-1}(\mathcal{H})}\mathcal{F}_d(H)$, then
\begin{equation}\label{w3}
\omega_3=d(k-1)\sum\limits_{H\in\mathcal{V}_{d-1}(\mathcal{H})}\sum\limits_{f\in\mathcal{F}_{d}(H)}\frac{\tau(f)\pi_{f}(\mathcal{A}_{\alpha}(\mathcal{H}))}{\prod\limits_{v\in V(f)}d^{+}(v)}.
\end{equation}

Substituting Eqs. (\ref{w1}), (\ref{w2}) and (\ref{w3}) into Eq. (\ref{laplacian3terms}), the expression for the $d$-th order $\mathcal{A}_{\alpha}$-spectral moment of $\mathcal{H}$ can be obtained directly.
\end{proof}

From Theorem \ref{T}, we can obtain the expressions of the first $k+1$ orders $\mathcal{A}_{\alpha}$-spectral moments of a $k$-uniform hypergraph.
Let $\mathcal{K}_{k+1}$ be the set of all complete $k$-uniform subhypergraphs with $k+1$ hyperedges of a $k$-uniform hypergraph $\mathcal{H}$.

\begin{cor}\label{firstk+1}
Let $k\ge2$ and $\mathcal{H}$ be a $k$-uniform hypergraph with degree sequence $d_1,d_2,\ldots,d_n$. Then
\begin{align*}
\mathrm{Tr}_d(\mathcal{A}_{\alpha}(\mathcal{H}))=
\begin{cases}
\phi(d),  ~ \quad  d=1,\ldots,k-1,\\
\phi(k)+(k-1)^{n-k}k^{k-1}(1-\alpha)^{k}|E(\mathcal{H})|, ~ \quad  d=k,\\
\phi(k+1)+(k+1)(k-1)^{n-k}\bigg( k^{k-2} \alpha (1-\alpha)^{k} \sum\limits_{i=1}^{n}d_i^2 & \\
~ \quad \quad \quad \quad +(1-\alpha)^{k+1}C_k|\mathcal{K}_{k+1}| \bigg), ~ \quad  d=k+1,
\end{cases}
\end{align*}
where $C_k$ is a constant depending only on $k$.
\end{cor}

Next, the expression of the $k+2$-nd order $\mathcal{A}_{\alpha}$-spectral moment of a $k$-uniform hypergraph is given.

\begin{thm}\label{k+2}
Let $k\ge2$ and $\mathcal{H}$ be a $k$-uniform hypergraph with degree sequence $d_1,d_2,\ldots,d_n$. Then
\begin{align*}
\mathrm{Tr}_{k+2}(\mathcal{A}_{\alpha}(\mathcal{H}))=&\phi(k+2)+(1-\alpha)^{k+2}\mathrm{Tr}_{k+2}(\mathcal{A}(\mathcal{H}))\\
+&(k+2)(k-1)^{n-k}k^{k-2}\alpha^{2}(1-\alpha)^{k}\sum\limits_{e\in E(\mathcal{H})}\left( \sum\limits_{\{i,j\}\subset e}d_id_j+\sum\limits_{i\in e}d_i^2 \right)\\
+&(k+2)(k-1)^{n-k-1}\alpha(1-\alpha)^{k+1}\sum\limits_{\mathcal{G}\in \mathcal{K}_{k+1}}\left( \sum\limits_{D\in R(\mathcal{G})}\tau(D) \right)\left( \sum\limits_{i\in V(\mathcal{G})}d_i \right),
\end{align*}
where $R(\mathcal{G})=\{D|~ D\cong D_f,~ f\in\mathcal{F}_{k+2}(\mathcal{G})\}$, $\tau(D)$ is the number of rooted spanning trees of the multi-digraph $D$ with a specified root.
\end{thm}

\begin{proof}
From Theorem \ref{T}, we have
\begin{equation}\label{k++2}
\begin{aligned}
\mathrm{Tr}_{k+2}(\mathcal{A}_{\alpha}(\mathcal{H}))=&\phi(k+2)+(1-\alpha)^{k+2}\mathrm{Tr}_{k+2}(\mathcal{A}(\mathcal{H}))\\
&+(k+2)(k-1)^{n}\sum\limits_{H\in\mathcal{V}_{k+1}(\mathcal{H})}\sum\limits_{f\in\mathcal{F}_{k+2}(H)}\frac{\tau(f)\pi_{f}(\mathcal{A}_{\alpha}(\mathcal{H}))}{\prod\limits_{v\in V(f)}d^{+}(v)}.
\end{aligned}
\end{equation}

For $e\in E(\mathcal{H})$, let $\mathcal{E}_{e}$ be the connected Veblen infragraph obtained by adding multiplicity $k$ to the hyperedge $e$.
Then $\mathcal{V}_{k+1}(\mathcal{H})=\{\mathcal{E}_{e}|\ e\in E(\mathcal{H})\}\cup\mathcal{K}_{k+1}$.

For a hyperedge $e=\{i_1,i_2,\ldots,i_k\}\in E(\mathcal{H})$ with $i_1<i_2<\ldots<i_k$, let $\alpha_{j}$ be a permutation of $e\setminus\{i_j\}$ and $\beta_{j}=i_j \cdots i_j\in[n]^{k-1}$ for $j=1,\ldots,k$.
The elements in $\mathcal{F}_{k+2}(\mathcal{E}_e)$ have the following two cases.\\
\textbf{Case 1.}
There are two elements in $f_1$ corresponding to the same diagonal entry $(\mathcal{A}_{\alpha})_{i_j \beta_j}$ of $\mathcal{A}_{\alpha}(\mathcal{H})$, where $j\in[k]$.
Without loss of generality, we assume $$f_1=(i_1\alpha_1,\ldots,i_j\alpha_j,i_j\beta_j,i_j\beta_j,i_{j+1}\alpha_{j+1},\ldots,i_k\alpha_k)\in\mathcal{F}_{k+2}(\mathcal{E}_e).$$
Then $$\pi_{f_1}(\mathcal{A}_{\alpha}(\mathcal{H}))=\alpha^{2}(1-\alpha)^{k} \left( \frac{1}{(k-1)!} \right)^{k}d_{i_j}^{2},~ \prod_{v\in V(f_1)}d_v^{+}=3(k-1)^k.$$
\textbf{Case 2.}
There are two elements in $f_2$ that correspond to the two different diagonal entries $(\mathcal{A}_{\alpha})_{i_j \beta_j}$ and $(\mathcal{A}_{\alpha})_{i_l \beta_l}$ of $\mathcal{A}_{\alpha}(\mathcal{H})$, where $j\ne l\in[k]$.
Without loss of generality, we assume $$f_2=(i_1\alpha_1,\ldots,i_j\alpha_j,i_j\beta_j,i_{j+1}\alpha_{j+1},\ldots,i_l\alpha_l, i_l\beta_l,i_{l+1}\alpha_{l+1},\ldots,i_k\alpha_k)\in\mathcal{F}_{k+2}(\mathcal{E}_e).$$
Then $$\pi_{f_2}(\mathcal{A}_{\alpha}(\mathcal{H}))=\alpha^{2}(1-\alpha)^{k}\left( \frac{1}{(k-1)!} \right)^{k}d_{i_j}d_{i_l},~ \prod_{v\in V(f_2)}d_v^{+}=4(k-1)^k.$$

As is known, the number of rooted spanning trees (with a specified root) of a complete digraph with $k$ vertices is $k^{k-2}$ \cite{Cayley}.
The digraphs obtained by removing all loops of $D_{f_1}$ and $D_{f_2}$ are complete digraphs with $k$ vertices, then $\tau(f_1)=\tau(f_2)=k^{k-2}$.
There are $3((k-1)!)^{k}$ (resp. $4((k-1)!)^{k}$) elements in $\mathcal{F}_{k+2}(\mathcal{E}_e)$ that have the same induced multi-digraph as $f_1$ (resp. $f_2$). Then
\begin{equation}\label{edges}
\begin{aligned}
&\sum\limits_{e\in E(\mathcal{H})}\sum\limits_{f\in\mathcal{F}_{k+2}(\mathcal{E}_{e})}\frac{\tau(f) \pi_{f}(\mathcal{A}_{\alpha}(\mathcal{H}))}{\prod\limits_{v\in V(f)}d_v^{+}}\\
=&(k-1)^{-k}k^{k-2}\alpha^{2}(1-\alpha)^{k}\sum\limits_{e\in E(\mathcal{H})}\left( \sum\limits_{i\in e}d_i^2+\sum\limits_{\{i,j\}\subset e}d_i d_j \right).
\end{aligned}
\end{equation}

For a subhypergraph $\mathcal{G}\in\mathcal{K}_{k+1}$ with the vertex set $V(\mathcal{G})=\{i_1,i_2,\ldots,i_{k+1}\}$ and the hyperedge set $E(\mathcal{G})=\{e_1,e_2,\ldots,e_{k+1}|~ i_j\in e_j,~ j=1,2,\ldots,k+1\}$, where $i_1<i_2<\ldots<i_{k+1}$, let $\alpha_{j}$ be a permutation of $e_j\setminus\{i_j\}$ and $\beta_{j}=i_{j}\cdots i_{j}\in[N]^{k-1}$ for $j=1,\dots,k+1$. The elements in $\mathcal{F}_{k+2}(\mathcal{G})$ have the only one case.
There is one element in $f$ corresponding to a diagonal entry of $\mathcal{A}_{\alpha}(\mathcal{H})$.
Without loss of generality, we assume $$f=(i_1\alpha_1,\ldots,i_j\alpha_j,i_j\beta_j,i_{j+1}\alpha_{j+1},\ldots,i_{k+1}\alpha_{k+1})\in\mathcal{F}_{k+2}(\mathcal{G}).$$
Then $$\pi_{f}(\mathcal{A}_{\alpha}(\mathcal{H}))=\alpha(1-\alpha)^{k+1}\left( \frac{1}{(k-1)!} \right)^{k+1}d_{i_j},~ \prod_{v\in V(f)}d_v^{+}=2(k-1)^{k+1}.$$
Since the number of spanning trees of a multi-digraph is irrelevant of the loops, we know that $\sum_{f\in\mathcal{F}_{k+2}(\mathcal{G})}\tau(f)=\sum_{D\in R(\mathcal{G})}\tau(D)$. There are $2((k-1)!)^{k+1}$ elements in $\mathcal{F}_{k+2}(\mathcal{G})$ that have the same induced multi-digraph as $f$. Then
\begin{equation}\label{simplices}
\begin{aligned}
&\sum\limits_{\mathcal{G}\in\mathcal{K}_{k+1}}\sum\limits_{f\in\mathcal{F}_{k+2}(\mathcal{G})}\frac{\tau(f) \pi_{f}(\mathcal{A}_{\alpha}(\mathcal{H}))}{\prod\limits_{v\in V(f)}d_v^{+}}\\
=&(k-1)^{-k-1}\alpha(1-\alpha)^{k+1}\sum\limits_{\mathcal{G}\in \mathcal{K}_{k+1}}\left( \sum\limits_{i\in V(\mathcal{G})}d_i \left( \sum\limits_{D\in R(\mathcal{G})}\tau(D) \right) \right).
\end{aligned}
\end{equation}

Substituting Eqs. (\ref{edges}) and (\ref{simplices}) into Eq. (\ref{k++2}), the expression of $\mathrm{Tr}_{k+2}(\mathcal{A}_{\alpha}(\mathcal{H}))$ is obtained.
\end{proof}

\section{Some transformations of hypergraphs}
In this section, we introduce some operations of moving hyperedges or hyperpaths on uniform hypergraphs and prove that the $S_{\alpha}$-order of hypergraphs is monotonic on these transformations.

For a $k$-uniform hypergraph $\mathcal{H}$, a vertex with degree one is called a \textit{cored vertex}, a hyperedge is called a \textit{pendant hyperedge} if it contains $k-1$ cored vertices.

\begin{lem}\label{1}
For a positive integer $k\ge2$, let $u$ and $v$ be two vertices of a $k$-uniform hypergraph $\mathcal{H}$ with $d_{\mathcal{H}}(u)<d_{\mathcal{H}}(v)+s$.
And $e_1,\ldots,e_s$ are all pendant hyperedges contains the vertex $u$.
The $k$-uniform hypergraph $\mathcal{H}^{'}$ is obtained by moving all pendant hyperedges $e_1,\ldots,e_s$ from $u$ to $v$, $\mathcal{H}^{'}$ is said to be a \textit{$\sigma$-transformation} of $\mathcal{H}$.
Then $\mathcal{H}\prec_{\alpha}\mathcal{H}^{'}$ for $0<\alpha<1$.
\end{lem}

\begin{proof}
If $0<\alpha<1$, then $\mathrm{Tr}_{t}(\mathcal{A}_{\alpha}(\mathcal{H}))=\mathrm{Tr}_{t}(\mathcal{A}_{\alpha}(\mathcal{H}^{'}))$ for $t=0,1$.
And
\begin{align*}
\sum_{v\in V(\mathcal{H})}d_{\mathcal{H}}(v)^2-\sum_{v\in V(\mathcal{H}^{'})}d_{\mathcal{H}^{'}}(v)^2
=&d_{\mathcal{H}}(u)^2+d_{\mathcal{H}}(v)^2-(d_{\mathcal{H}}(u)-s)^2-(d_{\mathcal{H}}(v)+s)^2\\
=&2s(d_{\mathcal{H}}(u)-d_{\mathcal{H}}(v)-s)<0.
\end{align*}
From Corollary \ref{firstk+1}, we know that $\mathrm{Tr}_{2}(\mathcal{A}_{\alpha}(\mathcal{H}))-\mathrm{Tr}_{2}(\mathcal{A}_{\alpha}(\mathcal{H}^{'}))<0$ for $0<\alpha<1$.
Then $\mathcal{H}\prec_{\alpha}\mathcal{H}^{'}$.
\end{proof}

Next, we introduce three transformations of silding hyperpath on uniform hypergraphs, and show that the transformed hypergraph comes before the original hypergraph in $S_{\alpha}$-order.

\begin{lem}\label{3}
Let $P^{(k)}$ be a $k$-uniform hyperpath.
The $k$-uniform hypergraph $\mathcal{H}$ is the coalescence of the hyperpath $P^{(k)}$ at one of its non-cored vertices with a $k$-uniform hypergraph $\mathcal{H}_0$ at its vertex $v_0$.
The $k$-uniform hypergraph $\mathcal{H}_{1}$ is obtained by moving $\mathcal{H}_0$ to a cored vertex of $P^{(k)}$, and $\mathcal{H}_1$ is said to be \textit{the first path-sliding transformation} of $\mathcal{H}$.
Then $\mathcal{H}_1\prec_{\alpha}\mathcal{H}$ for $k\ge3$ and $0<\alpha<1$.
\end{lem}

\begin{proof}
If $0<\alpha<1$, then $\mathrm{Tr}_{t}(\mathcal{A}_{\alpha}(\mathcal{H}_1))=\mathrm{Tr}_{t}(\mathcal{A}_{\alpha}(\mathcal{H}))$ for $t=0,1$.
It is obvious that $d_{\mathcal{H}}(v_0)=2+d_{\mathcal{H}_0}(v_0)>2$.
And
\begin{align*}
\sum_{v\in V(\mathcal{H}_1)}d_{\mathcal{H}_1}(v)^2-\sum_{v\in V(\mathcal{H})}d_{\mathcal{H}}(v)^2
&=2^2+(d_{\mathcal{H}}(v_0)-1)^2-1-d_{\mathcal{H}}(v_0)^2\\
&=2\left( 2-d_{\mathcal{H}}(v_0) \right)<0.
\end{align*}
From Corollary \ref{firstk+1}, we know that $\mathrm{Tr}_{2}(\mathcal{A}_{\alpha}(\mathcal{H}_1))-\mathrm{Tr}_{2}(\mathcal{A}_{\alpha}(\mathcal{H}))<0$ for $0<\alpha<1$.
Thus, $\mathcal{H}_1\prec_{\alpha}\mathcal{H}$.
\end{proof}

\begin{lem}\label{4}
Let $P_{r}^{(k)}$ be a $k$-uniform hyperpath with the hyperedge set $\{e_1^{(k)},\ldots,e_{r}^{(k)}\}$.
The $k$-uniform hypergraph $\mathcal{H}$ is the coalescence of the hyperpath $P_{r}^{(k)}$ at a cored vertex on the hyperedge $e_{l}^{(k)}$ with a $k$-uniform hypergraph $\mathcal{H}_0$ at its vertex $v_0$, where $2\le l\le \lceil\frac{r}{2}\rceil$ and $r\ge3$.
The $k$-uniform hypergraph $\mathcal{H}_{2}$ is obtained by moving $\mathcal{H}_0$ to a cored vertex on $e_{l-1}^{(k)}$, and $\mathcal{H}_2$ is said to be \textit{the second path-sliding transformation} of $\mathcal{H}$.
Then $\mathcal{H}_2\prec_{\alpha}\mathcal{H}$ for $k\ge3$ and $0<\alpha<1$.
\end{lem}

\begin{proof}
Let $|V(\mathcal{H})|=n$.
After the second path-sliding transformation, $\mathcal{H}$ and $\mathcal{H}_{2}$ have the same degree sequence.
We abbreviate $d_{\mathcal{H}_0}(v_0)$ as $d_0$.

When $l=2$, all complete $k$-uniform subhypergraphs with $k+1$ hyperedges of $\mathcal{H}$ and $\mathcal{H}_2$ both come from $\mathcal{H}_0$.
From Corollary \ref{firstk+1}, $\mathrm{Tr}_t(\mathcal{A}_{\alpha}(\mathcal{H}_2))=\mathrm{Tr}_t(\mathcal{A}_{\alpha}(\mathcal{H}))$ for $t=0,1,\ldots,k+1$.
Since all connected Veblen infragraphs of $\mathcal{H}$ and $\mathcal{H}_2$ with $k+2$ hyperedges are same, from Lemma 13 in \cite{ref10}, $\mathrm{Tr}_{k+2}(\mathcal{A}(\mathcal{H}_2))=\mathrm{Tr}_{k+2}(\mathcal{A}(\mathcal{H}))$.
Let $f_{l}^{(k)}$ denote the hyperedge of $\mathcal{H}_2$ which corresponds to $e_{l}^{(k)}$ in $\mathcal{H}$ for $l=1,\ldots,r$.
From Theorem \ref{k+2}, we have
\begin{align*}
&\mathrm{Tr}_{k+2}(\mathcal{A}_{\alpha}(\mathcal{H}_2))-\mathrm{Tr}_{k+2}(\mathcal{A}_{\alpha}(\mathcal{H}))\\
=& (k+2)(k-1)^{n-k}k^{k-2} \alpha^{2}(1-\alpha)^{k} \Bigg(\sum\limits_{l=1}^{2}\bigg(\sum\limits_{\{i,j\}\subset f^{(k)}_l}d_{\mathcal{H}_2}(i)d_{\mathcal{H}_2}(j)+\sum\limits_{i\in f^{(k)}_l}d_{\mathcal{H}_2}(i)^{2}\bigg) \\&-\sum\limits_{l=1}^{2}\bigg(\sum\limits_{\{i,j\}\subset e^{(k)}_l}d_{\mathcal{H}}(i)d_{\mathcal{H}}(j)+\sum\limits_{i\in e^{(k)}_l}d_{\mathcal{H}}(i)^{2}\bigg)\Bigg) \\
=& (k+2)(k-1)^{n-k}k^{k-2} \alpha^{2}(1-\alpha)^{k} \Bigg( \left( 2\binom{k-2}{2}+8(k-2)+k(d_0+1)+(d_0+1)^2+16 \right) \\&-\left( \binom{k-1}{2}+\binom{k-3}{2}+3(k-1)+5(k-3)+(k+1)(d_0+1)+(d_0+1)^2+16 \right) \Bigg)\\
=&-(k+2)(k-1)^{n-k}k^{k-2}\alpha^{2}(1-\alpha)^{k}d_0<0,
\end{align*}
where $0<\alpha<1$.
Hence, $\mathcal{H}_2\prec_{\alpha}\mathcal{H}$ for $l=2$ and $0<\alpha<1$.

Now we assume $3\le l\le \lceil\frac{r}{2}\rceil$ and $r\ge5$.
It is obvious that $\mathcal{V}_{d}(\mathcal{H})=\mathcal{V}_{d}(\mathcal{H}_2)$ and $\pi_{f}(\mathcal{A}_{\alpha}(\mathcal{H}))=\pi_{f}(\mathcal{A}_{\alpha}(\mathcal{H}_2))$ for any $f\in\mathcal{F}_d(H)$ with $H\in\mathcal{V}_{d}(\mathcal{H})$ and $d=1,\ldots,(l+1)k-1$.
From Theorem \ref{T}, we know that $\mathrm{Tr}_t(\mathcal{A}_{\alpha}(\mathcal{H}))=\mathrm{Tr}_t(\mathcal{A}_{\alpha}(\mathcal{H}_2))$ for $t=0,1,\ldots,(l+1)k-1$.
Note that the only difference between $\mathcal{V}_{(l+1)k}(\mathcal{H})$ and $\mathcal{V}_{(l+1)k}(\mathcal{H}_2)$ is the number of the subhypergraph $P_{l+1}^{(k)}$.
From Lemma 13 in \cite{ref10}, we have
\begin{align*}
&\mathrm{Tr}_{(l+1)k}(\mathcal{A}_{\alpha}(\mathcal{H}_2))-\mathrm{Tr}_{(l+1)k}(\mathcal{A}_{\alpha}(\mathcal{H}))\\
=&(1-\alpha)^{(l+1)k}\left(\mathrm{Tr}_{(l+1)k}(\mathcal{A}(\mathcal{H}_2))-\mathrm{Tr}_{(l+1)k}(\mathcal{A}(\mathcal{H}))\right)\\
=&(1-\alpha)^{(l+1)k}(l+1)k(k-1)^{n}C_{P_{l+1}^{(k)}}\left( N_{\mathcal{H}_2}(P_{l+1}^{(k)})-N_{\mathcal{H}}(P_{l+1}^{(k)}) \right)\\
=&-(l+1)k(k-1)^{n}(1-\alpha)^{(l+1)k}C_{P_{l+1}^{(k)}}d_0<0,
\end{align*}
where $0<\alpha<1$.
Hence, $\mathcal{H}_2\prec_{\alpha}\mathcal{H}$ for $3\le l\le \lceil\frac{r}{2}\rceil$ and $0<\alpha<1$.
\end{proof}

\begin{lem}\label{5}
Let $P_{r}^{(k)}$ be a $k$-uniform hyperpath with the hyperedge set $\{e_1^{(k)},\ldots,e_{r}^{(k)}\}$, where $e_{i}^{(k)}=\{u_{i-1},w_{i,1},\ldots,w_{i,k-2},u_i\}$ and $w_{i,1},\ldots,w_{i,k-2}$ are cored vertices for $i\in[r]$.
The $k$-uniform hypergraph $\mathcal{H}$ is the coalescence of the hyperpath $P_{r}^{(k)}$ at the vertex $u_l$ with a $k$-uniform hypergraph $\mathcal{H}_0$ at its vertex $v_0$, where $1\le l\le \lfloor\frac{r}{2}\rfloor$ and $r\ge2$.
The $k$-uniform hypergraph $\mathcal{H}_{3}$ is obtained by moving $\mathcal{H}_0$ form $u_l$ to $u_{l-1}$, and $\mathcal{H}_3$ is said to be \textit{the third path-sliding transformation} of $\mathcal{H}$.
Then $\mathcal{H}_{3}\prec_{\alpha}\mathcal{H}$ for $k\ge2$ and $0<\alpha<1$.
\end{lem}

\begin{proof}
Let $|V(\mathcal{H})|=n$.
We abbreviate $d_{\mathcal{H}_0}(v_0)$ as $d_0$.

First we assume $l=1$. If $0<\alpha<1$, it is obvious that  $\mathrm{Tr}_t(\mathcal{A}_{\alpha}(\mathcal{H}))=\mathrm{Tr}_t(\mathcal{A}_{\alpha}(\mathcal{H}_3))$ for $i=0,1$.
And
\begin{align*}
\mathrm{Tr}_2(\mathcal{A}_{\alpha}(\mathcal{H}_3))-\mathrm{Tr}_2(\mathcal{A}_{\alpha}(\mathcal{H}))
=-2(k-1)^{n-1}\alpha^{2}d_0<0.
\end{align*}
Thus $\mathcal{H}_{3}\prec_{\alpha}\mathcal{H}$ for $l=1$ and $0<\alpha<1$.

Next, we assume $l\ge2$ and $r\ge4$. It is obvious that $\mathcal{H}$ and $\mathcal{H}_3$ have the same degree sequence.
When $l=2$, $\mathrm{Tr}_{t}(\mathcal{A}_{\alpha}(\mathcal{H}))=\mathrm{Tr}_t(\mathcal{A}_{\alpha}(\mathcal{H}_3))$ for $i=0,1,\ldots,k+1$.
Let $f_{l}^{(k)}$ denote the hyperedge of $\mathcal{H}_3$ which corresponds to $e_{l}^{(k)}$ in $\mathcal{H}$ for $l=1,\ldots,r$.
Similarity to Lemma \ref{4}, from Theorem \ref{k+2}, we have
\begin{align*}
&\mathrm{Tr}_{k+2}(\mathcal{A}_{\alpha}(\mathcal{H}_3))-\mathrm{Tr}_{k+2}(\mathcal{A}_{\alpha}(\mathcal{H}))\\
=&(k+2)(k-1)^{n-k}k^{k-2} \alpha^2(1-\alpha)^{k} \Bigg(\sum\limits_{l=1}^{3}\bigg(\sum\limits_{\{i,j\}\subset f^{(k)}_l}d_{\mathcal{H}_3}(i)d_{\mathcal{H}_3}(j)+\sum\limits_{i\in f^{(k)}_l}d_{\mathcal{H}_3}(i)^{2}\bigg) \\
&-\sum\limits_{l=1}^{3}\bigg(\sum\limits_{\{i,j\}\subset e^{(k)}_l}d_{\mathcal{H}}(i)d_{\mathcal{H}}(j)+\sum\limits_{i\in e^{(k)}_l}d_{\mathcal{H}}(i)^{2}\bigg)\Bigg) \\
=&-(k+2)(k-1)^{n-k}k^{k-2}\alpha^{2}(1-\alpha)^{k}d_0<0,
\end{align*}
where $0<\alpha<1$.
Thus $\mathcal{H}_{3}\prec_{\alpha}\mathcal{H}$ for $l=2$ and $0<\alpha<1$.

When $3\le l\le \lfloor\frac{r}{2}\rfloor$ and $r\ge6$.
Similarity to Lemma \ref{4}, we know that $\mathrm{Tr}_t(\mathcal{A}_{\alpha}(\mathcal{H}))=\mathrm{Tr}_t(\mathcal{A}_{\alpha}(\mathcal{H}_3))$ for $t=0,1,\ldots,(l+1)k-1$ and the only difference between $\mathcal{V}_{(l+1)k}(\mathcal{H})$ and $\mathcal{V}_{(l+1)k}(\mathcal{H}_3)$ is the number of the subhypergraph $P_{l+1}^{(k)}$.
From Lemma 13 in \cite{ref10}, we have
\begin{align*}
&\mathrm{Tr}_{(l+1)k}(\mathcal{A}_{\alpha}(\mathcal{H}_2))-\mathrm{Tr}_{(l+1)k}(\mathcal{A}_{\alpha}(\mathcal{H}))\\
=&(1-\alpha)^{(l+1)k}\left(\mathrm{Tr}_{(l+1)k}(\mathcal{A}(\mathcal{H}_2))-\mathrm{Tr}_{(l+1)k}(\mathcal{A}(\mathcal{H}))\right)\\
=&(l+1)k(k-1)^{n}(1-\alpha)^{(l+1)k}C_{P_l^{(k)}}\left( N_{\mathcal{H}_3}(P_{l+1}^{(k)})-N_{\mathcal{H}}(P_{l+1}^{(k)}) \right)\\
=&-(l+1)k(k-1)^{n}(1-\alpha)^{(l+1)k}C_{P_{l+1}^{(k)}}d_0<0,
\end{align*}
where $0<\alpha<1$.
Then $\mathcal{H}_{3}\prec_{\alpha}\mathcal{H}$ for $3\le l\le \lfloor\frac{r}{2}\rfloor$ and $0<\alpha<1$.
\end{proof}

\section{The $S_{\alpha}$-order in linear unicyclic hypergraphs}
The \textit{girth} of a hypergraph is the minimum length of the cycles in it.
A hypergraph is said to be \textit{linear} if any two different hyperedges intersect into at most one vertex.
Note that a $k$-uniform unicyclic hypergraph is linear if and only if its girth greater than or equal to $3$.
In this section, we give the first and the last hypergraphs in $S_{\alpha}$-order among all $k$-uniform linear unicyclic hypergraphs with given girth.
And the first, the second, the last and the second last hypergraphs in $S_{\alpha}$-order among all $k$-uniform linear unicyclic hypergraphs are also characterized.

Let $C_{g}^{(k)}\odot S_{m-g}^{(k)}$ be the $k$-uniform linear unicyclic hypergraph with $m\ge3$ hyperedges and girth $g\ge3$ obtained from the hypercycle $C_g^{(k)}$ by adding $m-g$ pendant hyperedges to a non-cored vertex of $C_{g}^{(k)}$.

\begin{thm}\label{Ung}
Among all $k$-uniform linear unicyclic hypergraphs with $m\ge3$ hyperedges and girth $g(3\le g\le m)$, the last hypergraph in $S_{\alpha}$-order is $C_{g}^{(k)}\odot S_{m-g}^{(k)}$ for $k\ge2$ and $0<\alpha<1$.
\end{thm}

\begin{proof}
For any $k$-uniform linear unicyclic hypergraph $\mathcal{H}$ with $m$ hyperedges and girth $g$, it is obvious that $C_g^{(k)}$ is a subhypergraph of $\mathcal{H}$.
Repeating $\sigma$-transformation, $\mathcal{H}$ can be changed into a linear unicyclic hypergraph obtained from the hypercycle $C_g^{(k)}$ by adding $m-g$ pendant hyperedges to one of non-cored vertices of $C_g^{(k)}$, that is, $C_{g}^{(k)}\odot S_{m-g}^{(k)}$. From Lemma \ref{1},
we know that $\mathcal{H}\prec_{\alpha}C_{g}^{(k)}\odot S_{m-g}^{(k)}$.
Therefore, $C_{g}^{(k)}\odot S_{m-g}^{(k)}$ is the last hypergraph in $S_{\alpha}$-order among all $k$-uniform linear unicyclic hypergraphs with $m$ hyperedges and girth $g$.
\end{proof}

\begin{thm}\label{un3}
Among all $k$-uniform linear unicyclic hypergraphs with $m\ge3$ hyperedges, the last hypergraph in $S_{\alpha}$-order is $C_{3}^{(k)}\odot S_{m-3}^{(k)}$ for $k\ge2$ and $0<\alpha<1$.
\end{thm}

\begin{proof}
Note that the $\mathcal{A}_{\alpha}$-spectral moments $\mathrm{Tr}_0(\mathcal{A}_{\alpha})$ and $\mathrm{Tr}_1(\mathcal{A}_{\alpha})$ are same for all $k$-uniform linear unicyclic hypergraphs with $m$ hyperedges.
Let $\mathcal{U}_{m}=\{C_{g}^{(k)}\odot S_{m-g}^{(k)}|~ 3\le g\le m\}$.
From Theorem \ref{Ung}, we know that $C_{g}^{(k)}\odot S_{m-g}^{(k)}$ is the last hypergraph among all $k$-uniform linear unicyclic hypergraphs with $m$ hyperedges and girth $g$.
Then the last hypergraph in $S_{\alpha}$-order among all $k$-uniform linear unicyclic hypergraphs with $m$ hyperedges is the last hypergraph in $\mathcal{U}_{m}$. And
\begin{align*}
&\mathrm{Tr}_{2}(\mathcal{A}_{\alpha}(C_{g}^{(k)}\odot S_{m-g}^{(k)}))\\
=&(k-1)^{m(k-1)-1} \alpha^2 \left( (2+m-g)^2+4(g-1)+g(k-2)+(m-g)(k-1) \right)\\
=&(k-1)^{m(k-1)-1} \alpha^2 \left( g^2-(2m+1)g+m(m+k+3) \right).
\end{align*}
Let $f(x)=x^2-(2m+1)x+m(m+k+3)$ for $3\le x\le m$. We have $f^{'}(x)=2(x-m)-1\le0$ for $3\le x\le m$, then $f(x)\le f(3)$ for $3\le x\le m$.
Hence, $\mathrm{Tr}_{2}(\mathcal{A}_{\alpha}(C_{g}^{(k)}\odot S_{m-g}^{(k)}))\le \mathrm{Tr}_{2}(\mathcal{A}_{\alpha}(C_{3}^{(k)}\odot S_{m-3}^{(k)}))$ for $3\le g\le m$ with the equality if and only if $g=3$.
Then $U\preceq_{\alpha}C_{3}^{(k)}\odot S_{m-3}^{(k)}$ for $U\in\mathcal{U}_{m}$ with the equality if and only if $U\cong C_{3}^{(k)}\odot S_{m-3}^{(k)}$.
Therefore, $C_{3}^{(k)}\odot S_{m-3}^{(k)}$ is the last hypergraph in $S_{\alpha}$-order among all $k$-uniform linear unicyclic hypergraphs with $m$ hyperedges.
\end{proof}

Let $C^{(k)}_{3;n_1,n_2,n_3}$ be the $k$-uniform linear unicyclic hypergraph obtained from the hypercycle $C_{3}^{(k)}$ by adding $n_1,~ n_2,~ n_3$ pendant hyperedges to three non-cored vertices of $C_{3}^{(k)}$, respectively.

\begin{thm}\label{2ndUm}
Among all $k$-uniform linear unicyclic hypergraphs with $m\ge3$ hyperedges, the second last hypergraph in $S_{\alpha}$-order is $C^{(k)}_{3;m-4,1,0}$ for $k\ge2$ and $0<\alpha<1$.
\end{thm}

\begin{proof}
Let $U^{'}$ denote the hypergraph obtained by adding a pendant hyperedge to a vertex of $C_{3}^{(k)}\odot S_{m-4}^{(k)}$ whose degree is not $m-2$.
From Theorem \ref{un3} and the $\sigma$-transformation of hypergraphs, we know that the second last linear unicyclic hypergraph might be $U^{'}$.
If the pendant hyperedge is added to a vertex of degree $2$, then $U^{1}$ is the unicyclic hypergraph $C^{(k)}_{3;m-4,1,0}$.
It is obvious the $\mathcal{A}_{\alpha}$-spectral moments $\mathrm{Tr}_0(\mathcal{A}_{\alpha})$ and $\mathrm{Tr}_1(\mathcal{A}_{\alpha})$ are same for all $k$-uniform linear unicyclic hypergraphs with $m$ hyperedges.
For all $U^{'}$ which is obtained by adding a pendant hyperedge to a cored vertex of $C_{3}^{(k)}\odot S_{m-4}^{(k)}$, we have $\mathrm{Tr}_{2}(\mathcal{A}_{\alpha}(C^{(k)}_{3;m-4,1,0}))>\mathrm{Tr}_{2}(\mathcal{A}_{\alpha}(U^{'}))$, then $U^{'}\prec_{\alpha}C^{(k)}_{3;m-4,1,0}$. Therefore, $C^{(k)}_{3;m-4,1,0}$ is the second last hypergraph in $S_{\alpha}$-order among all $k$-uniform linear unicyclic hypergraphs with $m$ hyperedges.
\end{proof}

Let $\mathcal{U}_{m,g}$ be the set of all $k$-uniform linear unicyclic hypergraphs with $m$ hyperedges and girth $g$.
Let $\mathbb{U}_{m,g}$ be the set of the hypergraphs in $\mathcal{U}_{m,g}$ satisfies the degree of each vertex is at most 2.
Then we have the following conclusion.

\begin{lem}\label{Uu}
For any $U\in\mathbb{U}_{m,g}$ and $U^{'}\in\mathcal{U}_{m,g}\setminus\mathbb{U}_{m,g}$, $U\prec_{\alpha}U^{'}$ for $0<\alpha<1$.
\end{lem}

\begin{proof}
For any $k$-uniform linear unicyclic hypergraph $U^{'}\in\mathcal{U}_{m,g}\setminus\mathbb{U}_{m,g}$, by using the inverse $\sigma$-transformation and the first path-sliding transformation repeatedly, we can change $U^{'}$ into a $k$-uniform linear unicyclic hypergraph $U$ in $\mathbb{U}_{m,g}$.
From Lemma \ref{1} and Lemma \ref{3}, we know that $U\prec_{\alpha}U^{'}$.
\end{proof}

Let $C_{g}^{(k)}\cdot P_{m-g}^{(k)}$ be the coalescence of the hypercycle $C_{g}^{(k)}$ at one of its cored vertices with the hyperpath $P_{m-g}^{(k)}$ at one of its pendant vertices.
If $g=m$, then the unicyclic hypergraph $C_{m}^{(k)}\cdot P_{0}^{(k)}$ is the hypercycle $C_{m}^{(k)}$.

\begin{thm}\label{cgpn-g}
Among all $k$-uniform linear unicyclic hypergraphs with $m\ge3$ hyperedges and girth $g(3\le g\le m)$, the first hypergraph in $S_{\alpha}$-order is $C_{g}^{(k)}\cdot P_{m-g}^{(k)}$ for $k\ge3$ and $0<\alpha<1$.
\end{thm}

\begin{proof}
For any $k$-uniform linear unicyclic hypergraph $U\in\mathbb{U}_{m,g}$, by using the second path-sliding transformation repeatedly, $U$ can be changed into a $k$-uniform linear unicyclic hypergraph $U^{'}$ obtained by identifying the pendant vertices of some hyperpaths and some cored vertices of $C_{g}^{(k)}$.
From Lemma \ref{4}, we know that $U^{'}\prec_{\alpha}U$.
And $U^{'}$ and $C_{g}^{(k)}\cdot P_{m-g}^{(k)}$ have the same degree sequence, then the first $k+1$ orders $\mathcal{A}_{\alpha}$-spectral moments of $U^{'}$ and $C_{g}^{(k)}\cdot P_{m-g}^{(k)}$ are same.
Since $U^{'}$ might has more than one hyperedges with at least three vertices of degree 2,
from Theorem \ref{k+2},
it can be obtained that $\mathrm{Tr}_{k+2}(\mathcal{A}_{\alpha}(C_{g}^{(k)}\cdot P_{m-g}^{(k)}))\le\mathrm{Tr}_{k+2}(\mathcal{A}_{\alpha}(U^{'}))$ with the equality if and only if $U^{'}\cong C_{g}^{(k)}\cdot P_{m-g}^{(k)}$. Hence, $C_{g}^{(k)}\cdot P_{m-g}^{(k)}\preceq_{\alpha}U$ for any $U\in\mathbb{U}_{m,g}$ with the equality if and only if $U\cong C_{g}^{(k)}\cdot P_{m-g}^{(k)}$. It follows that $C_{g}^{(k)}\cdot P_{m-g}^{(k)}$ is the first hypergraph in $\mathbb{U}_{m,g}$. From Lemma \ref{Uu}, we know that $C_{g}^{(k)}\cdot P_{m-g}^{(k)}$ is the first hypergraph in $S_{\alpha}$-order among all $k$-uniform linear unicyclic hypergraphs with $m$ hyperedges and girth $g$.
\end{proof}

\begin{thm}\label{firstcycle}
Among all $k$-uniform linear unicyclic hypergraphs with $m\ge3$ hyperedges, the first hypergraph in $S_{\alpha}$-order is the hypercycle $C_{m}^{(k)}$ for $k\ge3$ and $0<\alpha<1$.
\end{thm}

\begin{proof}
Let $\mathbb{U}_{m}=\{C_{g}^{(k)}\cdot P_{m-g}^{(k)}|\ 3\le g\le m\}$.
From Theorem \ref{cgpn-g}, we know that $C_{g}^{(k)}\cdot P_{m-g}^{(k)}$ is the first hypergraph among all $k$-uniform linear unicyclic hypergraphs with $m$ hyperedges and girth $g$.
Then the first hypergraph in $S_{\alpha}$-order among all $k$-uniform linear unicyclic hypergraphs with $m$ hyperedges is the first in $\mathbb{U}_{m}$.
It is obvious that the $\mathcal{A}_{\alpha}$-spectral moments $\mathrm{Tr}_0(\mathcal{A}_{\alpha}),\mathrm{Tr}_1(\mathcal{A}_{\alpha}),\ldots,\mathrm{Tr}_{k+1}(\mathcal{A}_{\alpha})$ are same for all hypergraphs in $\mathbb{U}_m$.
When $3\le g\le m-1$ and $0<\alpha<1$, from Theorem \ref{k+2}, we have
$$\mathrm{Tr}_{k+2}(\mathcal{A}_{\alpha}(C_{m}^{(k)}))-\mathrm{Tr}_{k+2}(\mathcal{A}_{\alpha}(C_{g}^{(k)}\cdot P_{m-g}^{(k)}))=-(k+2)(k-1)^{m(k-1)-k}k^{k-2}\alpha^2(1-\alpha)^k<0,$$
i.e., $\mathrm{Tr}_{k+2}(\mathcal{A}_{\alpha}(C_{m}^{(k)}))<\mathrm{Tr}_{k+2}(\mathcal{A}_{\alpha}(C_{g}^{(k)}\cdot P_{m-g}^{(k)}))$ for $3\le g\le m-1$. Then $C_{m}^{(k)}\prec_{\alpha}U$ for all $U\in \mathbb{U}_{m}\setminus\{C_{m}^{(k)}\}$.
Hence, $C_{m}^{(k)}$ is the first hypergraph in $S_{\alpha}$-order among all $k$-uniform linear unicyclic hypergraphs with $m$ hyperedges.
\end{proof}

\begin{thm}
Among all $k$-uniform linear unicyclic hypergraphs with $m\ge3$ hyperedges, the second hypergraph in $S_{\alpha}$-order is $C_{m-1}^{(k)}\cdot P_{1}^{(k)}$ for $k\ge3$ and $0<\alpha<1$.
\end{thm}

\begin{proof}
From Theorem \ref{cgpn-g} and Theorem \ref{firstcycle}, we know that the second hypergraph in $S_{\alpha}$-order among all $k$-uniform linear unicyclic hypergraphs with $m$ hyperedges is the first in $\mathbb{U}_{m}\setminus\{C_{m}^{(k)}\}$.
For any $U_1, U_2\in\mathbb{U}_{m}\setminus\{C_{m}^{(k)}\}$ and $d=1,\ldots,2k+1$, we can know that $\mathcal{V}_{d}(U_1)=\mathcal{V}_{d}(U_2)$ and $\pi_{f}(\mathcal{A}_{\alpha}(U_1))=\pi_{f}(\mathcal{A}_{\alpha}(U_2))$ for any $f\in\mathcal{F}_{d}(H)$ and $H\in\mathcal{V}_{d-1}(U_1)$.
The first $2k+1$ orders spectral moments and the degree sequences of all hypergraphs in $\mathbb{U}_{m}\setminus\{C_{m}^{(k)}\}$ are same.
From Theorem \ref{T}, we know that $\mathrm{Tr}_{d}(\mathcal{A}_{\alpha}(U_1))=\mathrm{Tr}_{d}(\mathcal{A}_{\alpha}(U_2))$ for any $U_1, U_2\in\mathbb{U}_{m}\setminus\{C_{m}^{(k)}\}$ and $d\in[2k+1]$.

Next, we compare the $2k+2$-nd order $\mathcal{A}_{\alpha}$-spectral moments of hypergraphs in $\mathbb{U}_{m}\setminus\{C_{m}^{(k)}\}$.
For any $U\in\mathbb{U}_{m}\setminus\{C_{m}^{(k)},C_{m-1}^{(k)}\cdot P_{1}^{(k)}\}$, we can see that $\mathrm{Tr}_{2k+2}(\mathcal{A}(U))=\mathrm{Tr}_{2k+2}(\mathcal{A}(C_{m-1}^{(k)}\cdot P_{1}^{(k)}))$.
Since $U$ and $C_{m-1}^{(k)}\cdot P_{1}^{(k)}$ have the same degree sequence and the differences between $\pi_{f}(\mathcal{A}_{\alpha}(U))$ and $\pi_{f}(\mathcal{A}_{\alpha}(C_{m-1}^{(k)}\cdot P_{1}^{(k)}))$ are the terms obtained from $f\in\mathcal{F}_{2k+2}(P_2^{(k)})$, from Theorem \ref{T}, we have
\begin{align*}
&\mathrm{Tr}_{2k+2}(\mathcal{A}_{\alpha}(C_{m-1}^{(k)}\cdot P_{1}^{(k)}))-\mathrm{Tr}_{2k+2}(\mathcal{A}_{\alpha}(U))\\
=&(2k+2)(k-1)^{m(k-1)}\sum\limits_{f\in\mathcal{F}_{2k+2}(P_2^{(k)})}\left( \frac{\tau(f)\pi_{f}(\mathcal{A}_{\alpha}(C_{m-1}^{(k)}\cdot P_{1}^{(k)}))}{\prod\limits_{v\in V(f)}d^{+}(v)} - \frac{\tau(f)\pi_{f}(\mathcal{A}_{\alpha}(U))}{\prod\limits_{v\in V(f)}d^{+}(v)} \right).
\end{align*}

Given a $k$-uniform hypergraph $\mathcal{H}$, for $P_{2}^{(k)}\subset \mathcal{H}$ with the vertex set $V(P_{2}^{(k)})=\{i_1,\ldots,i_{2k-1}\}$ and the hyperedge set $E(P_2^{(k)})=\{ e_1, e_2 \}$, where $i_1<\cdots<i_{2k-1}$, $e_1=\{i_1,\ldots,i_k\}$ and $e_2=\{i_k,\ldots,i_{2k-1}\}$.
Let $\alpha_j$ be a permutation of $e_1 \setminus\{i_j\}$ for $j=1,\ldots,k-1$ and $\alpha_l$ be a permutation of $e_2 \setminus\{i_l\}$ for $l=k+1,\ldots,2k-1$.
Let $\alpha_k$ be a permutation of $e_1 \setminus\{i_k\}$ and $\alpha_{k+1}$ be a permutation of $e_2 \setminus\{i_k\}$.
Let $\beta_j=i_j\cdots i_j\in[N]^{k-1}$ for $j=1,\ldots,2k-1$, where $N=m(k-1)$.
The elements in $\mathcal{F}_{2k+2}(P_{2}^{(k)})$ have the following four cases.\\
\textbf{Case 1.} Let $f_1=(i_1\alpha_1,\ldots,i_k\alpha_k,i_k\beta_k,i_k\beta_k,i_k\alpha_{k+1},\ldots,i_{2k-1}\alpha_{2k-1})$.
Then $$\pi_{f_1}(\mathcal{A}_{\alpha}(\mathcal{H}))=\alpha^2(1-\alpha)^{2k}\left( \frac{1}{(k-1)!} \right)^{2k}d_{i_k}^{2},~ \sum_{v\in V(f_1)}d_v^{+}=4(k-1)^{2k-1}.$$
And there are $12((k-1)!)^{2k}$ elements in $\mathcal{F}_{2k+2}(P_{2}^{(k)})$ that have the same induced multi-digraph as $f_1$.\\
\textbf{Case 2.} Let $f_2=(i_1\alpha_1,\ldots,i_j\alpha_j,i_j\beta_j,i_j\beta_j,i_{j+1}\alpha_{j+1},\ldots,i_k\alpha_k,i_k\alpha_{k+1},\ldots,i_{2k-1}\alpha_{2k-1})$, where $j\in[2k-1]\setminus\{k\}$. Then $$\pi_{f_2}(\mathcal{A}_{\alpha}(\mathcal{H}))=\alpha^2(1-\alpha)^{2k}\left( \frac{1}{(k-1)!} \right)^{2k}d_{i_j}^{2},~ \sum_{v\in V(f_2)}d_v^{+}=6(k-1)^{2k-1}.$$
And there are $6((k-1)!)^{2k}$ elements in $\mathcal{F}_{2k+2}(P_{2}^{(k)})$ that have the same induced multi-digraph as $f_2$.\\
\textbf{Case 3.} Let $f_3=(i_1\alpha_1,\ldots,i_j\alpha_j,i_j\beta_j,i_{j+1}\alpha_{j+1},\ldots,i_k\alpha_k,i_k\beta_k,i_k\alpha_{k+1},\ldots,i_{2k-1}\alpha_{2k-1})$, where $j\in[2k-1]\setminus\{k\}$. Then $$\pi_{f_3}(\mathcal{A}_{\alpha}(\mathcal{H}))=\alpha^2(1-\alpha)^{2k}\left( \frac{1}{(k-1)!} \right)^{2k}d_{i_j}d_{i_k},~ \sum_{v\in V(f_3)}d_v^{+}=6(k-1)^{2k-1}.$$
And there are $12((k-1)!)^{2k}$ elements in $\mathcal{F}_{2k+2}(P_{2}^{(k)})$ that have the same induced multi-digraph as $f_3$.\\
\textbf{Case 4.} Let $f_4=(i_1\alpha_1,\ldots,i_j\alpha_j,i_j\beta_j,i_{j+1}\alpha_{j+1},\ldots,i_l\alpha_l,i_l\beta_l,i_{l+1}\alpha_{l+1},\ldots,i_k\alpha_k, i_k\alpha_{k+1},\\ \ldots,i_{2k-1}\alpha_{2k-1})$, where $j\ne l\in[2k-1]\setminus\{k\}$. Then $$\pi_{f_4}(\mathcal{A}_{\alpha}(\mathcal{H}))=\alpha^2(1-\alpha)^{2k}\left( \frac{1}{(k-1)!} \right)^{2k}d_{i_j}d_{i_l},~ \sum_{v\in V(f_4)}d_v^{+}=8(k-1)^{2k-1}.$$
And there are $8((k-1)!)^{2k}$ elements in $\mathcal{F}_{2k+2}(P_{2}^{(k)})$ that have the same induced multi-digraph as $f_4$.

It can be seen that the digraph obtained by removing all loops of $D_{f_i}$ is the digraph obtained by intersecting two complete digraphs with $k$ vertices at a common vertex, then $\tau(f_i)=k^{2k-4}$ for $i\in[4]$.

For any $U\in\mathbb{U}_{m}\setminus\{C_{m}^{(k)},C_{m-1}^{(k)}\cdot P_{1}^{(k)}\}$, by the vertex-hyperedge incidence relations of hypergraphs $C_{m-1}^{(k)}\cdot P_{1}^{(k)}$ and $U$, from the above statements,
when $0<\alpha<1$, we have
\begin{align*}
&\sum\limits_{f\in\mathcal{F}_{2k+2}(P_2^{(k)})}\left( \frac{\tau(f)\pi_{f}(\mathcal{A}_{\alpha}(C_{m-1}^{(k)}\cdot P_{1}^{(k)}))}{\prod\limits_{v\in V(f)}d^{+}(v)} - \frac{\tau(f)\pi_{f}(\mathcal{A}_{\alpha}(U))}{\prod\limits_{v\in V(f)}d^{+}(v)} \right).\\
=&(k-1)^{1-2k}(k-1)^{2k-4} \alpha^{2}(1-\alpha)^{2k} \Bigg( \sum\limits_{P_{2}^{k}\subset C_{m-1}^{(k)}\cdot P_{1}^{(k)}}\bigg( \sum\limits_{i\in V(P_{2}^{(k)})}d_i^2+\sum\limits_{i,j\in V(P_{2}^{(k)})}d_i d_j \bigg)\\
&-\sum\limits_{P_{2}^{k}\subset U}\bigg( \sum\limits_{i\in V(P_{2}^{(k)})}d_i^2+\sum\limits_{i,j\in V(P_{2}^{(k)})}d_i d_j \bigg) \Bigg)\\
=&(k-1)^{1-2k}(k-1)^{2k-4} \alpha^{2}(1-\alpha)^{2k} \Bigg( \left( 2\binom{2k-4}{2}+14(2k-4)+48 \right)\\
&-\left( \binom{2k-5}{2}+\binom{2k-3}{2}+9(2k-5)+5(2k-3)+52 \right) \Bigg)\\
=&-(k-1)^{1-2k}(k-1)^{2k-4}\alpha^{2}(1-\alpha)^{2k}<0.
\end{align*}
Thus, $\mathrm{Tr}_{2k+2}(\mathcal{A}_{\alpha}(C_{m-1}^{(k)}\cdot P_{1}^{(k)}))<\mathrm{Tr}_{2k+2}(\mathcal{A}_{\alpha}(U))$ for any $U\in\mathbb{U}_{m}\setminus\{C_{m}^{(k)}\}$ with $U\neq C_{m-1}^{(k)}\cdot P_{1}^{(k)}$.
It follows that $C_{m-1}^{(k)}\cdot P_{1}^{(k)}\preceq_{\alpha}U$ for all $U\in\mathbb{U}_{m}\setminus\{C_{m}^{(k)}\}$ with the equality if and only if $U\cong C_{m-1}^{(k)}\cdot P_{1}^{(k)}$.
Then $C_{m-1}^{(k)}\cdot P_{1}^{(k)}$ is the first hypergraph in $\mathbb{U}_{m}\setminus\{C_{m}^{(k)}\}$.
Therefore, $C_{m-1}^{(k)}\cdot P_{1}^{(k)}$ is the second hypergraph in $S_{\alpha}$-order among all $k$-uniform linear unicyclic hypergraphs with $m$ hyperedges.
\end{proof}

\section{The $S_{\alpha}$-order in hypertrees}
In this section, we give the first, the second, the last and the second last hypergraphs in $S_{\alpha}$-order among all $k$-uniform hypertrees.
And the last hypergraph among all $k$-uniform hypertrees with given diameter is characterized.

Let $\mathcal{T}_{m}$ be the set of all $k$-uniform hypertrees with $m$ hyperedges.
Let $\mathbb{T}_{m}$ be the set of the hypertrees in $\mathcal{T}_{m}$ satisfies the degree of each vertex is at most $2$.
Then we have the following conclusion.

\begin{lem}\label{Tt}
For any $T\in\mathbb{T}_m$ and $T^{'}\in\mathcal{T}_{m}\setminus\mathbb{T}_{m}$, $T\prec_{\alpha}T^{'}$ for $0<\alpha<1$.
\end{lem}

\begin{proof}
For any hypertree $T^{'}\in\mathcal{T}_{m}\setminus\mathbb{T}_{m}$, by using the inverse $\sigma$-transformation
and the first path-sliding transformation repeatedly, we can change $T^{'}$ into a $k$-uniform hypertree $T$ in $\mathbb{T}_m$.
From Lemma \ref{1} and Lemma \ref{3},
we know that $T\prec_{\alpha}T^{'}$.
\end{proof}

\begin{thm}\label{firsttree}
Among all $k$-uniform hypertrees with $m$ hyperedges, the first hypertree in $S_{\alpha}$-order is the hyperpath $P_m^{(k)}$ for $k\ge3$ and $0<\alpha<1$.
\end{thm}

\begin{proof}
From Lemma \ref{Tt}, the first hypertree in $\mathcal{T}_{m}$ is the first in $\mathbb{T}_{m}$.
For any hypertree $T\neq P_{m}^{(k)}$ in $\mathbb{T}_m$, by using the second path-sliding transformation repeatedly, $T$ can be changed into $P_m^{(k)}$.
From Lemma \ref{4}, we know that $P_{m}^{(k)}\prec_{\alpha}T$ for any $T\in\mathbb{T}_{m}\setminus\{P_{m}^{(k)}\}$.
Then $P_m^{(k)}$ is the first hypergraph in $S_{\alpha}$-order among all $k$-uniform hypertrees with $m$ hyperedges.
\end{proof}

Let $P_{m-1}$ be the path of length $m-1$ with the vertex set $V(P_{m-1})=\{v_1,\ldots,v_{m}\}$ and the edge set $E(P_{m-1})=\{e_i=\{v_i,v_{i+1}\}|\ i=1,\ldots,m-1\}$.
We form a $k$-uniform hypertree $F_{m,k}$ obtained from the hyperpath $P_{m-1}^{(k)}$ by attaching a pendant hyperedge to a cored vertex on the hyperedge $e_2^{(k)}$.

\begin{thm}\label{secondtree}
Among all $k$-uniform hypertrees with $m$ hyperedges, the second hypertree in $S_{\alpha}$-order is $F_{m,k}$ for $k\ge3$ and $0<\alpha<1$.
\end{thm}

\begin{proof}
From Lemma \ref{Tt} and Theorem \ref{firsttree}, we know that the second hypertree in $\mathcal{T}_m$ is the first in $\mathbb{T}_{m}\setminus\{P_m^{(k)}\}$. For any hypertree $T$ in $\mathbb{T}_m\setminus\{P_m^{(k)}\}$, repeating the second path-sliding transformation,
$T$ can be changed into $F_{m,k}$. From Lemma \ref{4}, we have $F_{m,k}\preceq_{\alpha}T$ for any $T\in\mathbb{T}_{m}\setminus\{P_m^{(k)}\}$ with equality if and only if $T\cong F_{m,k}$. Thus, $F_{m,k}$ is the second hypertree in $S_{\alpha}$-order among all $k$-uniform hypertrees with $m$ hyperedges.
\end{proof}

A $k$-uniform starlike hypertree, denoted by $S_{n_1,\ldots,n_m}^{(k)}$, is obtained from $m$ hyperpaths of length $n_1,\ldots,n_m$ by sharing a common vertex, where $m\ge1$ and $n_i\ge1$ for $i\in[m]$. When $m=1$, it is a $k$-uniform hyperpath $P_{n_1}^{(k)}$. When $n_i=1$ for each $i\in[m]$, it is a $k$-uniform hyperstar with $m$ hyperedges, denoted by $S_m^{(k)}$.

\begin{thm}\label{lasttree}
Among all $k$-uniform hypertrees with $m$ hyperedges, the last hypertree in $S_{\alpha}$-order is the hyperstar $S_{m}^{(k)}$ for $k\ge2$ and $0<\alpha<1$.
\end{thm}

\begin{proof}
For any $k$-uniform hypertree $T$ with $m$ hyperedges, repeating $\sigma$-transformation, $T$ can be changed into $S_m^{(k)}$. From Lemma \ref{1},
we know that $T\preceq_{\alpha}S_m^{(k)}$ for any $T\in\mathcal{T}_m$ with equality if and only if $T\cong S_m^{(k)}$.
Hence, $S_{m}^{(k)}$ is the last hypertree in $S_{\alpha}$-order among all $k$-uniform hypertrees with $m$ hyperedges.
\end{proof}

The following theorem gives the last hypergraph in $S_{\alpha}$-order among all $k$-uniform hypertrees with given diameter.

\begin{thm}\label{treelastd}
Among all $k$-uniform hypertrees with $m$ hyperedges and diameter $d(2\le d\le m)$, the last hypertree in $S_{\alpha}$-order is $S_{t,d-t,1,\ldots,1}^{(k)}(t=\lfloor\frac{d}{2}\rfloor)$ for $k\ge2$ and $0<\alpha<1$.
\end{thm}

\begin{proof}
If $T$ is a $k$-uniform hypertree with diameter $d$, then the hyperpath $P_d^{(k)}$ is a subhypergraph of $T$. Repeating $\sigma$-transformation, any $k$-uniform hypertree $T$ with $m$ hyperedges and diameter $d$ can be changed into a $k$-uniform hypertree $T_1$ obtained from $P_{d}^{(k)}$ by adding $m-d$ pendant hyperedges to a non-cored vertex of $P_{d}^{(k)}$. From Lemma \ref{1},
it is known that $T\prec_{\alpha}T_1$. By using the inverse third path-sliding transformation repeatedly, we can change $T_1$ into a $k$-uniform hypertree obtained from $P_{d}^{(k)}$ by adding $m-d$ pendant hyperedges to the non-cored vertex $v_{\lfloor \frac{d}{2} \rfloor+1}$ of $P_{d}^{(k)}$ i.e., $S_{\lfloor\frac{d}{2}\rfloor,d-\lfloor\frac{d}{2}\rfloor,1,\ldots,1}^{(k)}$. From Lemma \ref{5}, we know that $T_1\prec_{\alpha}S_{\lfloor\frac{d}{2}\rfloor,d-\lfloor\frac{d}{2}\rfloor,1,\ldots,1}^{(k)}$. Hence, $S_{\lfloor\frac{d}{2}\rfloor,d-\lfloor\frac{d}{2}\rfloor,1,\ldots,1}^{(k)}$ is the last hypergraph in $S_{\alpha}$-order among all $k$-uniform hypertrees with $m$ hyperedges and diameter $d$.
\end{proof}

Let $\mathcal{S}=\{S_{\lfloor\frac{d}{2}\rfloor,d-\lfloor\frac{d}{2}\rfloor,1,\ldots,1}^{(k)}|\ d=2,3,\ldots,m\}$.
When $d=2$, it is obvious that $S_{\lfloor\frac{d}{2}\rfloor,d-\lfloor\frac{d}{2}\rfloor,1,\ldots,1}^{(k)}\cong S_{m}^{(k)}$.
The hyperstar $S_{m}^{(k)}$ is not only the only $k$-uniform hypertree diameter $2$ but also the last hypertree in $S_{\alpha}$-order among all $k$-uniform hypertrees with $m$ hyperedges, then that $S_{m}^{(k)}$ is the last hypertree in $\mathcal{S}$. From Theorem \ref{treelastd}, we know that the second last hypertree in $\mathcal{T}_{m}$ must be the last in $\mathcal{S}\setminus\{S_{m}^{(k)}\}$.

\begin{thm}\label{lastsecondtree}
Among all $k$-uniform hypertrees with $m$ hyperedges, the second last hypertree in $S_{\alpha}$-order is the starlike hypertree $S_{1,2,1,\ldots,1}^{(k)}$ for $k\ge2$ and $0<\alpha<1$.
\end{thm}

\begin{proof}
Note that the $\mathcal{A}_{\alpha}$-spectral moments $\mathrm{Tr}_0(\mathcal{A}_{\alpha})$ and $\mathrm{Tr}_1(\mathcal{A}_{\alpha})$ are same for all $k$-uniform hypertrees in $\mathcal{S}\setminus\{S_{m}^{(k)}\}=\{S_{\lfloor\frac{d}{2}\rfloor,d-\lfloor\frac{d}{2}\rfloor,1,\ldots,1}^{(k)}|\ d=3,\ldots,m\}$. And
\begin{align*}
&\mathrm{Tr}_{2}(\mathcal{A}_{\alpha}(S_{\lfloor\frac{d}{2}\rfloor,d-\lfloor\frac{d}{2}\rfloor,1,\ldots,1}^{(k)}))\\
=&(k-1)^{m(k-1)-2}\alpha^2\left( (m-d+2)^2+4(d-2)+(m-d)(k-1)+d(k-2)+2 \right)\\
=&(k-1)^{m(k-1)-2}\alpha^2\left( d^2-(2m+1)d+(m+2)^2+m(k-1)-6 \right).
\end{align*}
Let $f(x)=x^2-(2m+1)x+(m+2)^2+m(k-1)-6$ for $3\le x\le m$. We have $f^{'}(x)=2(x-m)-1\le0$ for $3\le x\le m$, then $f(x)\le f(3)$ for $3\le x\le m$. Thus, $\mathrm{Tr}_{2}(\mathcal{A}_{\alpha}(S_{\lfloor\frac{d}{2}\rfloor,d-\lfloor\frac{d}{2}\rfloor,1,\ldots,1}^{(k)}))\le\mathrm{Tr}_{2}(\mathcal{A}_{\alpha}(S_{1,2,1,\ldots,1}^{(k)}))$ for $3\le d\le m$ with equality if and only if $d=3$ i.e., $S_{\lfloor\frac{d}{2}\rfloor,d-\lfloor\frac{d}{2}\rfloor,1,\ldots,1}^{(k)}\cong S_{1,2,1,\ldots,1}^{(k)}$. It follows that $S_{1,2,1,\ldots,1}^{(k)}$ is the last hypergraph in $\mathcal{S}\setminus\{S_{m}^{(k)}\}$.
Hence, $S_{1,2,1,\ldots,1}^{(k)}$ is the second last hypergraph in $S_{\alpha}$-order among all $k$-uniform hypertrees with $m$ hyperedges.
\end{proof}


\section{The extremums of $\mathcal{A}_{\alpha}$-spectral moments for hypertrees and unicyclic hypergraphs}
In this section, we give the extreme values of the $2$-nd order and the $k+2$-nd order $\mathcal{A}_{\alpha}$-spectral moments for linear unicyclic hypergraphs and hypertrees, respectively.

From Theorem \ref{Ung}, we can directly determine the hypergraph which has the largest $2$-nd order $\mathcal{A}_{\alpha}$-spectral moment among all linear unicyclic hypergraphs with given girth.
From Theorem \ref{un3} and Theorem \ref{2ndUm}, the hypergraphs with the largest and the second largest $2$-nd order $\mathcal{A}_{\alpha}$-spectral moments among all linear unicyclic hypergraphs can be directly determined, respectively.

\begin{thm}
Let $k\ge2$ and $0<\alpha<1$.
Among all $k$-uniform linear unicyclic hypergraphs with $m$ hyperedges and girth $g(3\le g\le m)$, $C_{g}^{(k)}\odot S_{m-g}^{(k)}$ has the largest $2$-nd order $\mathcal{A}_{\alpha}$-spectral moment.
Among all $k$-uniform linear unicyclic hypergraphs with $m$ hyperedges, $C_{3}^{(k)}\odot S_{m-3}^{(k)}$ has the largest $2$-nd order $\mathcal{A}_{\alpha}$-spectral moment, and $C_{3;m-4,1,0}^{(k)}$ has the second largest $2$-nd order $\mathcal{A}_{\alpha}$-spectral moment.
\end{thm}

The following theorem characterizes the hypergraph that has the smallest $k+2$-nd order $\mathcal{A}_{\alpha}$-spectral moment among all unicyclic hypergraphs (with given girth), which can be directly obtained from Theorem \ref{cgpn-g} and Theorem \ref{firstcycle}.

\begin{thm}
Let $k\ge3$ and $0<\alpha<1$.
Among all $k$-uniform linear unicyclic hypergraphs with $m$ hyperedges and girth $g(3\le g\le m)$, $C_{g}^{(k)}\cdot P_{m-g}^{(k)}$ has the smallest $k+2$-nd order $\mathcal{A}_{\alpha}$-spectral moment.
Among all $k$-uniform linear unicyclic hypergraphs with $m$ hyperedges, $C_{m}^{(k)}$ has the smallest $k+2$-nd order $\mathcal{A}_{\alpha}$-spectral moment.
\end{thm}

From Theorem \ref{lasttree} and Theorem \ref{lastsecondtree}, we can directly determine the hypertrees with the largest and the second largest $2$-nd order $\mathcal{A}_{\alpha}$-spectral moments among all hypertrees, respectively. From Theorem \ref{firsttree} and Theorem \ref{secondtree}, the hypertrees with the smallest and the second smallest $k+2$-nd order $\mathcal{A}_{\alpha}$-spectral moments among all hypertrees can be directly determined, respectively.

\begin{thm}
Let $0<\alpha<1$.
Among all $k$-uniform hypertrees with $m$ hyperedges, $S_{m}^{(k)}$ (resp. $S_{1,2,1,\ldots,1}^{(k)}$) has the largest (resp. the second largest) $2$-nd order $\mathcal{A}_{\alpha}$-spectral moments for $k\ge2$, and $P_{m}^{(k)}$ (resp. $F_{m,k}$) has the smallest (resp. the second smallest) $k+2$-nd order $\mathcal{A}_{\alpha}$-spectral moments for $k\ge3$.
\end{thm}

\vspace{3mm}

\noindent
\textbf{Acknowledgements}
\vspace{3mm}
\noindent

This work is supported by the National Natural Science Foundation of China (No. 12071097, 12371344), the Natural Science Foundation for The Excellent Youth Scholars of the Heilongjiang Province (No. YQ2022A002, YQ2024A009), the China Postdoctoral Science Foundation (No. 2024M761510) and the Fundamental Research Funds for the Central Universities.

\section*{References}
\bibliographystyle{plain}
\bibliography{spbib}
\end{spacing}
\end{document}